\documentclass{amsart}
\usepackage[alphabetic]{amsrefs}
\usepackage[all]{xy}
\usepackage[mathscr]{eucal}
\usepackage{graphicx,amssymb}
\CompileMatrices

\newtheorem{theorem}{Theorem}
\newtheorem{lemma}[theorem]{Lemma}
\newtheorem{proposition}[theorem]{Proposition}
\newtheorem{corollary}[theorem]{Corollary}

\theoremstyle{definition}
\newtheorem{definition}[theorem]{Definition}
\newtheorem{example}[theorem]{Example}

\theoremstyle{remark}
\newtheorem{remark}[theorem]{Remark}

\numberwithin{equation}{section}

\begin{document}

\newcommand{\Diff}{\operatorname{Diff}}
\newcommand{\Homeo}{\operatorname{Homeo}}
\newcommand{\Hom}{\operatorname{Hom}}
\newcommand{\Exp}{\operatorname{Exp}}
\newcommand{\Orb}{\operatorname{\textup{Orb}}}

\newcommand{\COrb}{\operatorname{\star\textup{Orb}}}
\newcommand{\GenericOrb}{\operatorname{\{\cdot\}\textup{Orb}}}
\newcommand{\CROrb}{\operatorname{\scriptscriptstyle{\blacklozenge}\scriptstyle\textup{Orb}}}
\newcommand{\ssslozenge}{\scriptscriptstyle{\blacklozenge}}
\newcommand{\ssstriangledown}{\scriptscriptstyle{\blacktriangledown}}

\newcommand{\Stwo}{\mbox{$\displaystyle S^2$}}
\newcommand{\Sn}{\mbox{$\displaystyle S^n$}}
\newcommand{\supp}{\operatorname{supp}}
\newcommand{\intr}{\operatorname{int}}
\newcommand{\kernel}{\operatorname{ker}} \newcommand{\A}{\mathbb{A}}
\newcommand{\B}{\mathbb{B}} \newcommand{\C}{\mathbb{C}}
\newcommand{\D}{\mathbb{D}} \newcommand{\E}{\mathbb{E}}
\newcommand{\F}{\mathbb{F}} \newcommand{\G}{\mathbb{G}}
\newcommand{\Hh}{\mathbb{H}} \newcommand{\I}{\mathbb{I}}
\newcommand{\J}{\mathbb{J}} \newcommand{\K}{\mathbb{K}}
\newcommand{\Ll}{\mathbb{L}} \newcommand{\M}{\mathbb{M}}
\newcommand{\N}{\mathbb{N}} \newcommand{\Oo}{\mathbb{O}}
\newcommand{\Pp}{\mathbb{P}} \newcommand{\Q}{\mathbb{Q}}
\newcommand{\R}{\mathbb{R}} \newcommand{\Ss}{\mathbb{S}}
\newcommand{\T}{\mathbb{T}} \newcommand{\U}{\mathbb{U}}
\newcommand{\V}{\mathbb{V}} \newcommand{\W}{\mathbb{W}}
\newcommand{\X}{\mathbb{X}} \newcommand{\Y}{\mathbb{Y}}
\newcommand{\Z}{\mathbb{Z}} \newcommand{\kk}{\mathbb{k}}
\newcommand{\orbify}[1]{\ensuremath{\mathcal{#1}}}
\newcommand{\starfunc}[1]{\ensuremath{{}_\star{#1}}}
\newcommand{\lozengefunc}[1]{\ensuremath{{}_{\scriptscriptstyle{\blacklozenge}}{#1}}}
\newcommand{\redfunc}[1]{\ensuremath{{}_\bullet{#1}}}

\newcommand{\OrbDiff}{\ensuremath{\Diff_{\Orb}}}
\newcommand{\RedOrbDiff}{\ensuremath{\Diff_{\textup{red}}}}
\newcommand{\COrbDiff}{\ensuremath{\Diff_{\COrb}}}
\newcommand{\CROrbDiff}{\ensuremath{\Diff_{\CROrb}}}
\newcommand{\OrbMaps}{\ensuremath{C_{\Orb}}}
\newcommand{\RedOrbMaps}{\ensuremath{C_{\textup{red}}}}
\newcommand{\COrbMaps}{\ensuremath{C_{\COrb}}}
\newcommand{\CROrbMaps}{\ensuremath{C_{\CROrb}}}
\newcommand{\OrbMapsGeneric}{\ensuremath{C_{\GenericOrb}}}
\newcommand{\Frechet}{Fr\'{e}chet\ }
\newcommand{\Frechetnospace}{Fr\'{e}chet}

\title[Orbifold Differential Topology]{Elementary Orbifold Differential Topology}

\author{Joseph E. Borzellino}
\address{Department of Mathematics, California Polytechnic State
  University, 1 Grand Avenue, San Luis Obispo, California 93407}
\email{jborzell@calpoly.edu}

\author{Victor Brunsden} \address{Department of Mathematics and
  Statistics, Penn State Altoona, 3000 Ivyside Park, Altoona,
  Pennsylvania 16601} \email{vwb2@psu.edu}


\date{\today} \commby{Editor} 
\begin{abstract}
Taking an elementary and straightforward approach, we develop the concept of a regular value for a smooth map $f:\orbify{O}\to\orbify{P}$ between smooth orbifolds $\orbify{O}$ and $\orbify{P}$. We show that Sard's theorem holds and that the inverse image of a regular value is a smooth full suborbifold of $\orbify{O}$. We also study some constraints that the existence of a smooth orbifold map imposes on local isotropy groups. As an application, we prove a Borsuk no retraction theorem for compact orbifolds with boundary and some obstructions to the existence of real-valued orbifold maps from local model orbifold charts. 
\end{abstract}

\maketitle

\section{Introduction}\label{IntroSection}

Inspired by the elementary and elegant treatment of differential topology found in J. Milnor's book \cite{MR1487640}, \emph{Topology from a differentiable viewpoint}, we generalize some of the fundamental material of that book to the category of smooth orbifolds in a manner that is elementary. 

\section{Smooth Orbifolds}\label{OrbifoldsSection}

Although there are many references for this background material, we will use our previous work \cites{MR2523149,BB_Stratified} as our standard reference. While much of what we discuss here works equally well for smooth $C^r$ orbifolds, to simplify the exposition, we restrict ourselves to smooth $C^\infty$ orbifolds. Throughout, the term \emph{smooth} means $C^\infty$. This results in no loss of generality \citelist{\cite{MR2523149}*{Proposition~3.11} \cite{Kankaanrinta}}. Note that the classical definition of orbifold given below is modeled on the definition in Thurston \cite{Thurston78} and that these orbifolds are referred to as {\em classical effective orbifolds} in \cite{MR2359514}.

\begin{definition}\label{orbifold}
  An $n$-dimensional  \emph{smooth orbifold} $\orbify{O}$,
  consists of a paracompact, Hausdorff topological space
  $X_\orbify{O}$ called the \emph{underlying space}, with the
  following local structure.  For each $x \in X_\orbify{O}$ and
  neighborhood $U$ of $x$, there is a neighborhood $U_x \subset U$, an
  open set $\tilde U_x$ diffeomorphic to $\R^n$, a finite group $\Gamma_x$ acting
  smoothly and effectively on $\tilde U_x$ which fixes $0\in\tilde
  U_x$, and a homeomorphism $\phi_x:\tilde U_x/\Gamma_x \to U_x$ with
  $\phi_x(0)=x$.  These actions are subject to the condition that for
  a neighborhood $U_z\subset U_x$ with corresponding $\tilde U_z \cong
  \R^n$, group $\Gamma_z$ and homeomorphism $\phi_z:\tilde
  U_z/\Gamma_z \to U_z$, there is a smooth embedding $\tilde\psi_{zx}:\tilde
  U_z \to \tilde U_x$ and an injective homomorphism
  $\theta_{zx}:\Gamma_z \to \Gamma_x$ so that $\tilde\psi_{zx}$ is
  equivariant with respect to $\theta_{zx}$ (that is, for
  $\gamma\in\Gamma_z, \tilde\psi_{zx}(\gamma\cdot \tilde
  y)=\theta_{zx}(\gamma)\cdot\tilde\psi_{zx}(\tilde y)$ for all
  $\tilde y\in\tilde U_z$), such that the following diagram commutes:
  \begin{equation*}
    \xymatrix{{\tilde U_z}\ar[rr]^{\tilde\psi_{zx}}\ar[d]&&{\tilde U_x}\ar[d]\\
      {\tilde U_z/\Gamma_z}\ar[rr]^>>>>>>>>>>{\psi_{zx}=\tilde\psi_{zx}/\Gamma_z}\ar[dd]^{\phi_z}&&{\tilde U_x/\theta_{zx}(\Gamma_z)\ar[d]}\\
      &&{\tilde U_x/\Gamma_x}\ar[d]^{\phi_x}\\
      {U_z}\ar[rr]^{\subset}&&{U_x} }
  \end{equation*}
\end{definition}

We will refer to the neighborhood $U_x$ or $(\tilde U_x,\Gamma_x)$ or
$(\tilde U_x,\Gamma_x, \rho_x, \phi_x)$ as an \emph{orbifold chart}, and write $U_x=\tilde U_x/\Gamma_x$. In the 4-tuple  notation, we are making explicit the representation $\rho_x:\Gamma_x\to\text{Diff}^\infty(\tilde U_x)$.
The \emph{isotropy group of $x$} is the group $\Gamma_x$. The
definition of orbifold implies that the germ of the action of $\Gamma_x$ in a
neighborhood of the origin of $\R^n$ is unique, so that by shrinking $\tilde U_x$ if necessary, $\Gamma_x$ is well-defined up to isomorphism. The
\emph{singular set} of \orbify{O} is the set of points
$x \in \orbify{O}$ with $\Gamma_x \ne \{ e\}$.
More detail can be found in \cite{MR2523149}.

\begin{definition}\label{OrbifoldWithBdyDef} A \emph{smooth orbifold with boundary} $\orbify{X}$, is an orbifold as in definition~\ref{orbifold} where one replaces the requirement that $\tilde U_x$ be diffeomorphic to $\R^n$ with the requirement that $\tilde U_x$ be diffeomorphic to $\R^n$ or $\R^n_+$, the closed upper half-space. The boundary $\partial\orbify{X}$ of $\orbify{X}$ consists of those points $x\in\orbify{X}$ where $\tilde U_x$ is diffeomorphic to $\R^n_+$. Throughout the rest of the article, we will use $\orbify{X}$ to denote a smooth orbifold with nonempty boundary.
\end{definition}

\subsection{Compact 1-dimensional orbifolds}\label{CompactOneOrbifolds} Using the classification of compact 1-dimensional manifolds, it is easy to classify all 1-dimensional compact connected orbifolds with or without boundary. There are four types: (a) the circle $S^1$, (b) the closed interval $[0,1]$ with trivial orbifold structure, (c) the closed interval $[0,1]$ with where $\{0\}$ is a singular point with $\Z_2$ isotropy, and (d) the closed interval $[0,1]$ with where both $\{0,1\}$  have $\Z_2$ isotropy. Thus, the compact 1-orbifolds must be finite unions of orbifolds of these types. See figure~\ref{Compact1Orbifolds}.

\begin{figure}[ht]
  \centering
  \includegraphics[width=360pt]{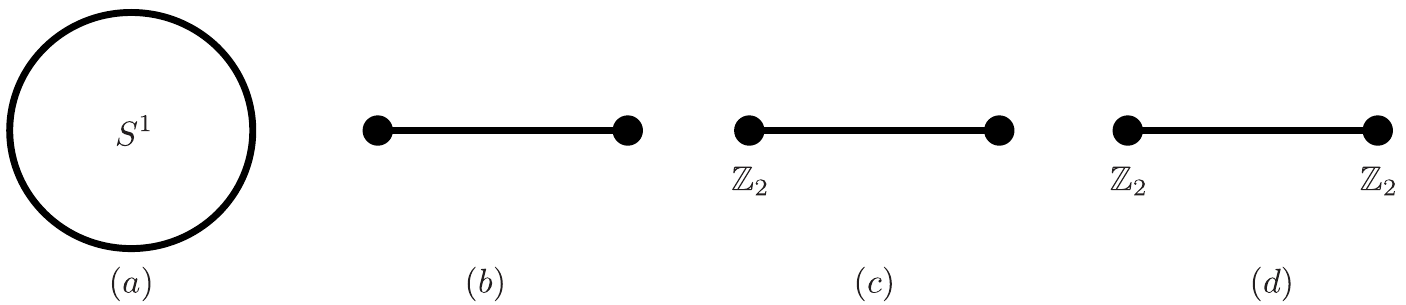}
  \caption{Compact connected 1-orbifolds}
  \label{Compact1Orbifolds}
\end{figure}

\subsection{Smooth Suborbifolds}\label{SuborbifoldSection}
The definition of suborbifold is somewhat subtle and we distinguish two types of suborbifolds.

\begin{definition}\label{SubOrbifold}
  An (embedded) \emph{suborbifold} \orbify{P} of an orbifold \orbify{O} consists
  of the following.
  \begin{enumerate}
  \item A subspace $X_{\orbify{P}}\subset X_{\orbify{O}}$ equipped
    with the subspace topology
  \item For each $x\in X_{\orbify{P}}$ and neighborhood $W$ of $x$ in
    $X_{\orbify{O}}$ there is an orbifold chart $(\tilde U_x,
    \Gamma_x, \rho_x, \phi_x)$ about $x$ in \orbify{O} with
    $U_x\subset W$, a subgroup $\Lambda_x \subset \Gamma_x$ of the
    isotropy group of $x$ in \orbify{O} and a $\rho_x(\Lambda_x)$
    invariant linear submanifold $\tilde V_x\subset \tilde U_x \cong \R^n$,
    so that $(\tilde V_x, \Lambda_x/\Omega_x,\rho_x\lvert_{\Lambda_x},\psi_x))$ is
    an orbifold chart for $\orbify{P}$ where $\Omega_x=\left\{\gamma\in\Lambda_x\mid \rho_x(\gamma)\lvert_{\tilde{V}_x}=\text{Id}\right\}$. (In particular, the \emph{intrinsic} isotropy subgroup at $x\in\orbify{P}$ is  $\Lambda_x/\Omega_x$), and
  \item
      $V_x = \psi_x(\tilde V_x/\rho_x(\Lambda_x))
      =U_x\cap X_{\orbify{P}}$
    is an orbifold chart for $x$ in $\orbify{P}$.
  \end{enumerate}

\end{definition}

\begin{remark}
  Originally, in \cite{Thurston78}, the notion of an $m$--suborbifold
  $\orbify{P}$ of an $n$--orbifold $\orbify{O}$ required
  $\orbify{P}$ to be locally modeled on $\R^m\subset\R^n$ modulo
  finite groups. That is, the local action on $\R^m$ is induced by the
  local action on $\R^n$. This is equivalent to adding the condition that $\Lambda_x = \Gamma_x$ at all $x$ in the underlying topological space of $\orbify{P}$. 
\end{remark}
Given this remark, we make the following definition:

\begin{definition}\label{FullSuborbifoldDef} $\orbify{P}\subset\orbify{O}$ is a \emph{full suborbifold} of $\orbify{O}$ if $\orbify{P}$ is a suborbifold with
$\Lambda_x=\Gamma_x$ for all $x\in\orbify{P}$.
\end{definition}

\begin{example}\label{FullSuborbifoldExample} Let $\orbify{Q}=\R/\Z_2$ be the smooth orbifold (without boundary)  where $\Z_2$ acts on $\R$ via $\gamma\cdot x=-x$. The underlying topological space $X_{\orbify{Q}}$ of $\orbify{Q}$ is $[0, \infty)$ and the isotropy subgroups are $\{1\}$ for $x\in (0, \infty)$ and $\Z_2$ for $x = 0$. Let $\orbify{O}=\orbify{Q}\times\orbify{Q}$ be the smooth product orbifold (without boundary). See \cite{MR2523149}*{Definiton~2.12}. The underlying space for $\orbify{O}$ can be identified with the closed first quadrant and the singular points of $\orbify{O}$ lie in one of three connected singular strata: the positive $x$ axis, the positive $y$ axis (corresponding to those points with $\Z_2$ isotropy), and the origin which has $\Z_2\times\Z_2$ isotropy. Then both $\orbify{Q}\times\{0\}$ and $\{0\}\times\orbify{Q}$ are full suborbifolds of $\orbify{O}$. On the other hand, the diagonal $\textup{diag}(\orbify{Q})=\{(x,x)\mid x\in\orbify{Q}\}\subset\orbify{O}$ is merely a suborbifold. See \cite{MR2523149}*{Example~2.15}.
\end{example}

\begin{example}\label{SuborbifoldExample}
  Let \orbify{O} be as in example~\ref{FullSuborbifoldExample}. Consider the circle $\orbify{S}\subset\orbify{O}$ of radius 1 centered at $(1,1)$. Then $\orbify{S}$ is a suborbifold of $\orbify{O}$ that is not a full suborbifold. To see this, just note that at the point $x=(1,0)\in\orbify{O}$ any lift of $\orbify{S}$ to $\tilde U_x\cong\R^2$ in a neighborhood of $x$, cannot be an invariant linear submanifold unless we choose $\Lambda_x=\{1\}$. In this case, we see that the intrinsic isotropy group of $\orbify{S}$ at $x$ is trivial which it must be since $\orbify{S}$ is actually a compact 1-dimensional \emph{manifold}. That is, a compact 1-dimensional orbifold with trivial orbifold structure.
\end{example}

\begin{remark} 
  Let $\orbify{P}\subset\orbify{O}$ be a suborbifold.  Note that even
  though a point $p\in X_{\orbify{P}}$ may be in the singular set of
  \orbify{O}, it need not be in the singular set of \orbify{P}.
\end{remark}

\section{Smooth Mappings Between Orbifolds}\label{OrbifoldMapNotions}

In the literature, there are four related definitions of maps between orbifolds which are based on the classical Satake-Thurston approach to orbifolds via atlases of orbifold charts. In this paper, we only need to use the notion of complete orbifold map. It is distinguished from the other notions of orbifold map in that we are going to keep track of all defining data. All other notions of orbifold map descend from the complete orbifold maps by forgetting information. It turns out that the results of this paper also follow using any of the four notions of orbifold map. This requires only an understanding on how these notions of orbifold map are related to one another. We point this out explicitly in our exposition below. More detail can be found in \cite{BB_Stratified} and in what follows we use the notation of \cite{MR2523149}*{Section 2}.

The original motivation for defining the notion of complete orbifold map was to make meaningful and well-defined certain geometric constructions involving orbifolds and their maps. The need to be careful in defining an adequate notion of orbifold map was already noted in the work of Moerdijk and Pronk \cite{MR1466622} and Chen and Ruan \cite{MR1950941} and was missing from Satake's original work on $V$-manifolds \cites{MR0079769,MR0095520}.

\subsection{Complete Orbifold Maps}\label{CompleteOrbiMapSubSection}
\begin{definition}\label{CompleteOrbiMap}  
  A $C^\infty$ \emph{complete orbifold map} $(f,\{\tilde f_x\},\{\Theta_{f,x}\})$ between
  smooth orbifolds $\orbify{O}$ and $\orbify{P}$ consists of the
  following:
  \begin{enumerate}
  \item A continuous map $f:X_{\orbify{O}}\to X_{\orbify{P}}$ of
    the underlying topological spaces.
  \item For each $y\in \orbify{O}$, a group homomorphism
    $\Theta_{f,y}:\Gamma_y\to\Gamma_{f(y)}$.
  \item A smooth $\Theta_{f,y}$-equivariant lift $\tilde f_y:\tilde
    U_y\to\tilde V_{f(y)}$ where
    $(\tilde U_y,\Gamma_y)$ is an
    orbifold chart at $y$ and $(\tilde V_{f(y)},\Gamma_{f(y)})$ is an orbifold chart at $f(y)$.  That
    is, the following diagram commutes:
    \begin{equation*}
      \xymatrix{{\tilde U_y}\ar[rr]^{\tilde f_y}\ar[d]&&{\tilde V_{f(y)}}\ar[d]\\
        {\tilde U_y}/\Gamma_y\ar[rr]^>>>>>>>>>>%
        {{\tilde f_y}/\Theta_{f,y}(\Gamma_y)}\ar[dd]&&{\tilde
          V_{f(y)}}/\Theta_{f,y}(\Gamma_y)\ar[d]\\
        &&{\tilde V_{f(y)}}/\Gamma_{f(y)}\ar[d]\\
        U_y\ar[rr]^{f}&&V_{f(y)}
      }
    \end{equation*}
\begin{sloppypar}
  \item[($\star 4$)] (Equivalence) Two complete orbifold maps $(f,\{\tilde f_x\},\{\Theta_{f,x}\})$ and
    $(g,\{\tilde g_x\},\{\Theta_{g,x}\})$ are considered equivalent if for each
    $x\in\orbify{O}_1$, $\tilde f_x=\tilde g_x$ as germs and $\Theta_{f,x}=\Theta_{g,x}$. That is,
    there exists an orbifold chart $(\tilde U_x,\Gamma_x)$ at $x$ such
    that ${\tilde f_x}\vert_{\tilde U_x}={\tilde g_x}\vert_{\tilde
      U_x}$ and $\Theta_{f,x}=\Theta_{g,x}$. Note that this implies that $f=g$.
\end{sloppypar}
  \end{enumerate}
  The set of smooth complete orbifold maps from $\orbify{O}$ to $\orbify{P}$ will be denoted by
$C^\infty_{\COrb}(\orbify{O}, \orbify{P})$. For $\orbify{O}$ compact (without boundary), $C^\infty_{\COrb}(\orbify{O}, \orbify{P})$ carries the structure of a smooth \Frechet manifold \cite{BB_Stratified}.
\end{definition}

\subsection{Regular and Critical Values}
\begin{definition} Let $\starfunc{f}=(f,\{\tilde f_x\},\{\Theta_{f,x}\}):\orbify{O}\to\orbify{P}$ be a smooth complete orbifold map between smooth orbifolds. A point $p\in\orbify{P}$ is a \emph{regular value} for $\starfunc{f}$ if $d\tilde f_x(\tilde x):T_{\tilde x}\tilde U_x\to T_{\tilde p}\tilde V_p$ is surjective for all $x\in f^{-1}(p)$. Otherwise, $p$ is a \emph{critical value} for $\starfunc{f}$. By convention, if $p\notin f(\orbify{O})$, then $p$ is a regular value.
\end{definition}

\begin{remark}\label{OrbMapNotionInvariant} Because all local lifts of an orbifold map $f:\orbify{O}\to\orbify{P}$ at $x$ differ from one another by the action of an element of $\Gamma_{f(x)}$ (which acts by diffeomorphisms on $\tilde{V}_{f(x)}$), it is clear that the notion of regular value is well-defined for any of the four notions of orbifold map.
\end{remark}

\section{Sard's Theorem and Preimage Theorem}

The local structure of a smooth orbifold is that of a quotient by a finite action by diffeomorphisms which is measure non-increasing. Hence the usual Sard's theorem for manifolds \cite{MR1487640} yields a Sard's theorem for smooth orbifolds.

\begin{theorem}[Sard's Theorem for Orbifolds] Let $f:\orbify{O}\to\orbify{P}$ be a (complete) smooth orbifold map. Then the set of critical values for $f$ has measure $0$ in $\orbify{P}$ and thus the set of regular values is everywhere dense in $\orbify{P}$.
\end{theorem}

We are ready to state our first main result which is the analogue of the so-called preimage theorem:

\begin{theorem}[Preimage Theorem for Orbifolds]\label{PreimageThm} Let $\orbify{O},\orbify{P}$ be smooth orbifolds (without boundary) with $\dim\orbify{O}\ge\dim\orbify{P}$. Let $f:\orbify{O}\to\orbify{P}$ be a (complete) smooth orbifold map and $p\in\orbify{P}$ a regular value for $f$. Then $f^{-1}(p)=\orbify{S}$ has the structure of a full, smooth suborbifold of dimension
$\dim(\orbify{S})=\dim(\orbify{O})-\dim(\orbify{P})$. Moreover, the local isotropy groups 
$\Gamma_{x,\orbify{S}}=\Gamma_{x,\orbify{O}}/G_{x,\orbify{O}}$ where 
$G_{x,\orbify{O}}=\{\gamma\in\Gamma_{x,\orbify{O}}\mid d\gamma\vert_{\textup{ker}(d\tilde f_x(\tilde x))}=\textup{Id}\}$.
\end{theorem}

\begin{proof} It suffices to work in a chart. For $x\in\orbify{S}$, $\tilde f^{-1}_x(\tilde p)$ is a submanifold $\tilde S_x$ of $\tilde U_x$ of dimension $\dim(\orbify{O})-\dim(\orbify{P})$ and $T_{\tilde x}\tilde S_x=\textup{ker}(d\tilde f_x(\tilde x))$, by the preimage theorem for manifolds  \cite{MR1487640}. The submanifold $\tilde S_x$ is $\Gamma_{x,\orbify{O}}$-invariant. To see this, let $\tilde y\in\tilde S_x$ and $\gamma\in\Gamma_{x,\orbify{O}}$. Then 
$$\tilde f_x(\gamma\cdot\tilde y)=\Theta_{f,x}(\gamma)\cdot\tilde f_x(\tilde y)=\Theta_{f,x}(\gamma)\cdot\tilde p=\tilde p$$
since $\Theta_{f,x}(\gamma)\in\Gamma_{p,\orbify{P}}$. Thus, $\gamma\cdot\tilde y\in\tilde S_x$ and we have shown that $\tilde S_x$ is $\Gamma_{x,\orbify{O}}$-invariant. Thus, a neighborhood of $x\in\orbify{S}$ can be realized as the quotient $\tilde S_x/\Gamma_{x,\orbify{O}}\cong\tilde S_x/(\Gamma_{x,\orbify{O}}/\Omega_{x,\orbify{O}})$, where $\Omega_{x,\orbify{O}}=\{\gamma\in\Gamma_{x,\orbify{O}}\mid\gamma\vert_{\tilde S_x}=\textup{Id}\}$. Since $\Gamma_{x,\orbify{O}}/\Omega_{x,\orbify{O}}$ acts effectively, we have shown that $\orbify{S}$ has the structure of a full suborbifold of $\orbify{O}$ with local isotropy groups $\Gamma_{x,\orbify{S}}=\Gamma_{x,\orbify{O}}/\Omega_{x,\orbify{O}}$. The Bochner-Cartan theorem \cites{MR0073104,MR2523149} implies that the smooth action of $\Gamma_{x,\orbify{O}}$ is smoothly conjugate to the linear action on $\tilde U_x$ given by the differential of the action. Since $\Gamma_{x,\orbify{S}}=\Gamma_{x,\orbify{O}}/\Omega_{x,\orbify{O}}$, the representation of 
$\Gamma_{x,\orbify{S}}$ given in the last statement of the theorem follows. 
\end{proof}

More generally, as in the case for manifolds we get a preimage theorem for orbifolds with boundary. We omit the proof.

\begin{theorem}[Preimage Theorem for Orbifolds with Boundary]\label{PreimageThmBdy} Let $\orbify{X}$ be a smooth orbifold with boundary and $\orbify{P}$ a smooth orbifold with $\dim\orbify{X}>\dim\orbify{P}$. Let $f:\orbify{X}\to\orbify{P}$ be a (complete) smooth orbifold map and $p\in\orbify{P}$ a regular value for $f$ and for the restriction $f\vert_{\partial\orbify{X}}$. Then $f^{-1}(p)=\orbify{S}$ has the structure of a full, smooth suborbifold with boundary of dimension
$\dim(\orbify{S})=\dim(\orbify{X})-\dim(\orbify{P})$. Moreover, the boundary $\partial(f^{-1}(p))$ is the intersection
$f^{-1}(p)\cap\partial\orbify{X}$.
\end{theorem}

\begin{remark}\label{PreimageThmInvariant}By remark~\ref{OrbMapNotionInvariant}, it follows that each of the results of this section also holds for any of the four notions of orbifold map. Specifically, because all local lifts of an orbifold map $f:\orbify{O}\to\orbify{P}$ at $x$ differ from one another by the action of an element of $\Gamma_{f(x)}$ (which acts by diffeomorphisms on $\tilde{V}_{f(x)}$), we have that $T_{\tilde x}\tilde S_x=\textup{ker}(d\tilde f_x(\tilde x))=\textup{ker}(d(\eta_x\cdot \tilde f_x)(\tilde x))$ for any $\eta_x\in\Gamma_{f(x)}$.
\end{remark}

\section{Implications of the Existence of Smooth Map Between Orbifolds}

Unsurprisingly, there are obstructions (which are manifested in the local orbifold chart structure) to the existence of a smooth map between orbifolds. In this section, we give the main tool we will use later. To avoid cumbersome notation, the induced action of $\gamma\in\Gamma_x$ on tangent vectors $\tilde v\in T_{\tilde x}{\tilde U_x}$ will be denoted by
left multiplication as well: $\gamma\cdot\tilde v=d\gamma_{\tilde{x}}(\tilde v)$ when convenient.

Let $\starfunc{f}=(f,\{\tilde f_x\},\{\Theta_{f,x}\}):\orbify{O}\to\orbify{P}$ be a smooth (complete) orbifold map. Let $K_x=\textup{ker}(d\tilde f_x(\tilde x))$ and $N_x=\textup{ker}\,\Theta_{f,x}\subset\Gamma_{x,\orbify{O}}$, a normal subgroup. For all $\tilde v\in T_{\tilde x}\tilde U_x$ and $\gamma\in N_x$ we have
$$d\tilde f_x(\tilde x)(\gamma\cdot\tilde v)=\Theta_{f,x}(\gamma)\cdot d\tilde f_x(\tilde x)(\tilde v)=%
d\tilde f_x(\tilde x)(\tilde v).$$
Thus, $\gamma\cdot\tilde v-\tilde v\in K_x$. In other words, for each $\gamma\in N_x$ we have a linear map
$A_\gamma=(\gamma-I)\in\textup{Hom}(T_{\tilde x}\tilde U_x,K_x)$. Here, $I$ denotes the identity map.

We have $\gamma\cdot\tilde v=(I+A_\gamma)\tilde v$ and thus, $(\gamma\delta)\cdot\tilde v=(I+A_{\gamma\delta})\tilde v$.
On the other hand, we have
\begin{align*}
(I+A_{\gamma\delta})\tilde v &=(\gamma\delta)\cdot\tilde v=\gamma\cdot(\delta\cdot\tilde v)=\gamma\cdot(I+A_\delta)\tilde v=\gamma\cdot\tilde v+\gamma\cdot A_\delta\tilde v\\
&=(I+A_\gamma)\tilde v+\gamma\cdot A_\delta\tilde v=(I+A_\gamma+\gamma\cdot A_\delta)\tilde v
\end{align*}
Also,
\begin{align*}
(I+A_{\gamma\delta})\tilde v &=(\gamma\delta)\cdot\tilde v=\gamma\cdot(\delta\cdot\tilde v)=(I+A_\gamma)(\delta\cdot\tilde v)=\delta\cdot\tilde v+A_\gamma(\delta\cdot\tilde v)\\
&=(I+A_\delta)\tilde v+A_\gamma(\delta\cdot\tilde v)=(I+A_\delta+A_\gamma\delta)\tilde v
\end{align*}
Similarly,
\begin{align*}
(I+A_{\gamma\delta})\tilde v &=(\gamma\delta)\cdot\tilde v=\gamma\cdot(\delta\cdot\tilde v)=(I+A_\gamma)(\delta\cdot\tilde v)=\delta\cdot\tilde v+A_\gamma(\delta\cdot\tilde v)\\
&=\delta\cdot\tilde v+A_\gamma(I+A_\delta)\tilde v=\delta\cdot\tilde v+A_\gamma\tilde v+A_\gamma A_\delta\tilde v\\
&=(I+A_\delta)\tilde v+A_\gamma\tilde v+A_\gamma A_\delta\tilde v=(I+A_\delta+A_\gamma+A_\gamma A_\delta)\tilde v
\end{align*}
We thus have three expressions for $A_{\gamma\delta}$:
\begin{align*}A_{\gamma\delta}&=
A_\gamma+\gamma\cdot A_\delta\\
&=A_\delta+A_\gamma\delta\cdot\\ 
&=A_\delta+A_\gamma+A_\gamma A_\delta
\end{align*}

\begin{proposition}\label{NxInvariantProjectionMap} With $\starfunc{f}$ and notation as above, for each $x\in\orbify{O}$, there exists an $N_x$-invariant linear projection $A_x\in\textup{Hom}_{N_x}(T_{\tilde x}\tilde U_x,K_x)$.
\end{proposition}
\begin{proof} Define the linear map $A\in\textup{Hom}(T_{\tilde x}\tilde U_x,K_x)$, by
$A=\dfrac{1}{|N_x|}\displaystyle\sum_{\delta\in N_x}A_\delta$. Then 
$\gamma\cdot A=\dfrac{1}{|N_x|}\displaystyle\sum_{\delta\in N_x}\gamma\cdot A_\delta=\dfrac{1}{|N_x|}\displaystyle\sum_{\delta\in N_x}(A_{\gamma\delta}-A_\gamma)=A-A_\gamma$. Therefore, $A_\gamma=A-\gamma\cdot A$. Similarly,
$A\delta\cdot = \dfrac{1}{|N_x|}\displaystyle\sum_{\gamma\in N_x}A_\gamma\delta\cdot=\dfrac{1}{|N_x|}\displaystyle\sum_{\gamma\in N_x}(A_{\gamma\delta}-A_\delta)=A-A_\delta$. Therefore, $A_\delta=A-A\delta\cdot$. Putting this together, we conclude that $\gamma\cdot A=A\gamma\cdot$ and thus, $A$ is $N_x$-invariant. To show that $A$ is a projection, we compute
\begin{align*}
A^2 &= \dfrac{1}{|N_x|^2}\displaystyle\sum_{\gamma\in N_x}\sum_{\delta\in N_x}A_\gamma A_\delta=
\dfrac{1}{|N_x|^2}\displaystyle\sum_{\gamma\in N_x}\sum_{\delta\in N_x}(A_{\gamma\delta}-A_\gamma-A_\delta)\\
&=\dfrac{1}{|N_x|^2}\displaystyle\sum_{\gamma\in N_x}|N_x|(A-A_\gamma-A)=-A
\end{align*}
Thus, $A_x=-A$ is the required $N_x$-invariant linear projection.
\end{proof}

\begin{lemma}\label{RestrictionLemma} For all $\tilde v\in\textup{ker}A_x$ and $\gamma\in N_x$, $\gamma\cdot\tilde v=\tilde v$. That is, $\gamma\vert_{\textup{ker}A_x}=\textup{Id}$.

\end{lemma}
\begin{proof}
Using proposition~\ref{NxInvariantProjectionMap}, since $A_x$ is a projection, the tangent space 
decomposes $T_{\tilde x}\tilde U_x=\textup{ker}A_x\oplus\textup{im}A_x$ and furthermore, since $A_x$ is $N_x$-invariant, so is this decomposition. For $\tilde v\in T_{\tilde x}\tilde U_x$ and $\gamma\in N_x$, we have $\gamma\cdot\tilde v-\tilde v=A_\gamma\tilde v=(A_x-\gamma\cdot A_x)\tilde v$. Thus, $\gamma\cdot\tilde v-\tilde v\in\textup{im}A_x$ since $\textup{im}A_x$ is $N_x$-invariant. If we further suppose that $\tilde v\in\textup{ker}A_x$, then since $\textup{ker}A_x$ is $N_x$-invariant, we must have $\gamma\cdot\tilde v-\tilde v\in\textup{ker}A_x\cap\textup{im}A_x=\{0\}$. This implies that $\gamma\cdot\tilde v=\tilde v$ for all $\tilde v\in\textup{ker}A_x$ and $\gamma\in N_x$.
\end{proof}

\begin{proposition}\label{NxFaithful} Let $\orbify{O},\orbify{P}$ be smooth orbifolds with $\dim\orbify{O}\ge\dim\orbify{P}$. Let $\starfunc{f}=(f,\{\tilde f_x\},\{\Theta_{f,x}\}):\orbify{O}\to\orbify{P}$ be a smooth (complete) orbifold map and $p\in\orbify{P}$ a regular value for $f$. Then there is a faithful representation of $N_x=\textup{ker}\,\Theta_{f,x}$ in $\Gamma_{x,\orbify{S}}$ where $\orbify{S}=f^{-1}(p)$ is the full, smooth suborbifold given by the preimage theorem~\ref{PreimageThm}.
\end{proposition}
\begin{proof}
Let $K_x=\textup{ker}(d\tilde f_x(\tilde x))$. By theorem~\ref{PreimageThm}, $\Gamma_{x,\orbify{S}}=\Gamma_{x,\orbify{O}}/G_{x,\orbify{O}}$ where 
$G_{x,\orbify{O}}=\{\gamma\in\Gamma_{x,\orbify{O}}\mid \gamma\vert_{K_x}=\textup{Id}\}$. Then $G_{x,\orbify{O}}\cap N_x=\{\textup{Id}\}$. For, if 
$\gamma\in G_{x,\orbify{O}}\cap N_x$, then by lemma~\ref{RestrictionLemma}, $\gamma\vert_{\textup{ker}A_x}=\textup{Id}$. Also, since 
$\gamma\vert_{K_x}=\textup{Id}$ and $\textup{im}A_x\subset K_x$, then $\gamma\vert_{\textup{im}A_x}=\textup{Id}$. Since $T_{\tilde x}\tilde U_x=\textup{ker}A_x\oplus\textup{im}A_x$, we conclude that $\gamma=\textup{Id}$. Consider the quotient homomorphism
$\Gamma_{x,\orbify{O}}\to\Gamma_{x,\orbify{O}}/G_{x,\orbify{O}}\cong\Gamma_{x,\orbify{S}}$ and restrict to the normal subgroup $N_x$:
$$N_x\to N_xG_{x,\orbify{O}}/G_{x,\orbify{O}}\cong N_x/(N_x\cap G_{x,\orbify{O}})\cong N_x.$$
From this we see that $N_x$ is faithfully represented in $\Gamma_{x,\orbify{S}}$.
\end{proof}

\begin{remark}It follows that each of the results of this section also holds for any of the four notions of orbifold map by our previous remarks~\ref{OrbMapNotionInvariant} and~\ref{PreimageThmInvariant}, and the observation that $N_x=\textup{ker}\,\Theta_{f,x}=\textup{ker}\,\eta_x\Theta_{f,x}\eta_x^{-1}$, for all $\eta_x\in\Gamma_{f(x)}$.
\end{remark}

\section{Applications}
In this section we give some applications of our results.

\begin{example}\label{RnGammaRnEx} Let $\Gamma$ be a finite group. Suppose that $\orbify{O}=\R^n/\Gamma$, with $\Gamma$ acting linearly on $\R^n$ and $\orbify{P}=\R^n$ (with the trivial orbifold structure). Let $\starfunc{f}=(f,\{\tilde f_x\},\{\Theta_{f,x}\}):\orbify{O}\to\orbify{P}$ be a smooth (complete) orbifold map. Assume $f(0)=p$. Then $p\in\orbify{P}$ is never a regular value. For otherwise,
$\Gamma$ would be forced to act effectively on 0-dimensional singleton by proposition~\ref{NxFaithful}, which is impossible.
\end{example}

\begin{example}\label{RnGammaRkEx} Let $\Gamma$ be a finite group. Suppose that $\orbify{O}=\R^n/\Gamma$, with $\Gamma$ acting linearly on $\R^n$ via an irreducible representation. Let $\orbify{P}=\R^k/\Gamma$ where $k<n$ and $\Gamma$  any effective action on $\R^k$. Let $\starfunc{f}=(f,\{\tilde f_x\},\{\Theta_{f,x}\}):\orbify{O}\to\orbify{P}$ be a smooth (complete) orbifold map. Assume $f(0)=p$. Then $p\in\orbify{P}$ is never a regular value. For, otherwise,
$\Gamma$ would be forced leave an $(n-k)$-dimensional subspace of $\R^n$ invariant by theorem~\ref{PreimageThm}, which cannot happen by our assumption of irreducibility of the action of $\Gamma$ on $\R^n$.
\end{example}

\begin{example}\label{SingularSetPostiveDim} 
Let $\starfunc{f}=(f,\{\tilde f_x\},\{\Theta_{f,x}\}):\orbify{O}\to\R$ be a smooth (complete) orbifold map where 
$\orbify{O}$ is a smooth $n$-dimensional orbifold (without boundary) and $\R$ has been given  the trivial orbifold structure. Suppose $p$ is a regular value of $\starfunc{f}$. Then $f^{-1}(p)=\orbify{S}$ is a full suborbifold of dimension $(n-1)$. For $x\in\orbify{S}$, we have $N_x=\Gamma_{x,\orbify{O}}$. Since $G_{x,\orbify{O}}\cap N_x=\{\textup{Id}\}$ (see proof of proposition~\ref{NxFaithful}), we have that 
$\Gamma_{x,\orbify{S}}=\Gamma_{x,\orbify{O}}$ and thus $\Gamma_{x,\orbify{O}}$ acts effectively on 
$K_x=\textup{ker}(d\tilde f_x(\tilde x))=T_{\tilde x}\tilde{\orbify{S}}_x\cong\R^{n-1}$. Since $\tilde v\in\textup{ker}A_x$ implies $\Gamma_{x,\orbify{O}}\cdot\tilde v=\tilde v$ by lemma~\ref{RestrictionLemma} and $\Gamma_{x,\orbify{O}}$ acts effectively on $K_x$, we see that $\textup{ker}A_x\cap K_x=\{0\}$. This implies that $K_x\subset\textup{im}A_x$ and since $\textup{im}A_x\subset K_x$ by definition, $K_x=\textup{im}A_x$ and hence $\textup{ker}A_x\cong\R$. Thus we have a $\Gamma_{x,\orbify{O}}$-invariant decomposition of the tangent space 
$T_{\tilde x}\tilde U_x=\textup{ker}A_x\oplus\textup{im}A_x=\R\oplus K_x$. In particular, again by lemma~\ref{RestrictionLemma}, the $\R$ factor of this decomposition is fixed by the action of $\Gamma_{x,\orbify{O}}$ and thus we conclude that $\Sigma_{\orbify{O}}(x)$, the connected component of the singular set of $\orbify{O}$ that contains $x$, must be empty or have dimension $\dim(\Sigma_{\orbify{O}}(x))\ge 1$.
\end{example}

\begin{example} As an application of example~\ref{SingularSetPostiveDim} we conclude that  if $\starfunc{f}:\orbify{O}\to\R$ is a smooth (complete) orbifold map and $p\in\R$ is a regular value, then $f^{-1}(p)$ cannot contain any isolated points in the singular set of $\orbify{O}$.
\end{example}

A generalization of K.~Borsuk's so-called \emph{No Retraction Theorem} \cites{Borsuk-1931,MR1487640} states that there is no smooth map from a compact manifold with boundary to its boundary that leaves the boundary fixed. We prove an analogue of this result for orbifolds. The following example shows that some extra assumptions are necessary in the orbifold case.

\begin{example}\label{1-OrbifoldRetraction} Let $\orbify{X}$ be the compact 1-orbifold with boundary of type (c) given in section~\ref{CompactOneOrbifolds}. Then a smooth (complete) orbifold map $\starfunc{f}:\orbify{X}\to\partial\orbify{X}$ with 
$\starfunc{f}\vert_{\partial\orbify{X}}=\textup{Id}$ is given by the constant map $x\mapsto 1=\partial\orbify{X}$.
\end{example}

\begin{theorem}\label{NoRetractionToBdy} Let $\orbify{X}$ be a smooth $n$-dimensional compact orbifold with boundary $\partial\orbify{X}$ and assume that the interior, $\textup{int}\,\orbify{X}$, does not have any codimension 1 singular strata. Then there is no smooth (complete) orbifold map $\starfunc{f}:\orbify{X}\to\partial\orbify{X}$ with $\starfunc{f}\vert_{\partial\orbify{X}}=\textup{Id}$.
\end{theorem}

\begin{proof} Suppose such $\starfunc{f}$ exists. By Sard's theorem there exists a regular value $p\in\partial\orbify{X}$. Furthermore, since the singular set of an orbifold is nowhere dense, we may further assume that $p$ is not in the singular set of $\partial\orbify{X}$. Therefore, by theorem~\ref{PreimageThmBdy}, $f^{-1}(p)=\orbify{S}$ is a full, smooth 1-orbifold with boundary and $\partial\orbify{S}=\orbify{S}\cap\orbify\partial\orbify{X}=\{p\}$ since $\starfunc{f}\vert_{\partial\orbify{X}}=\textup{Id}$. Because $\partial\orbify{S}$ consists of a single point, there must be a connected component 
$\orbify{S}_{\textup{c}}$ of $\orbify{S}$ isomorphic to a compact 1-orbifold of type (c). Consider the unique point $z\in\orbify{S}_{\textup{c}}\cap\textup{int}\,\orbify{X}$ where $\Gamma_{z,\orbify{S}_{\textup{c}}}=\mathbb{Z}_2$. Arguing as in example~\ref{SingularSetPostiveDim}, we can conclude that $\Gamma_{z,\orbify{X}}=\Gamma_{z,\orbify{S}_{\textup{c}}}=\mathbb{Z}_2$ and that we have a $\mathbb{Z}_2$-invariant decomposition of the tangent space $T_{\tilde z}\tilde U_z=\R^{n-1}\oplus\R$ which leaves the $\R^{n-1}$ factor fixed. This implies that the dimension of the singular stratum containing $z$ has codimension 1. By assumption, no such points $z\in\orbify{X}$ exist and we have our desired contradiction.
\end{proof}

\begin{corollary}[No Retraction Theorem for Orbifolds]\label{NoRetractionThmOrb} Let $\orbify{X}$ be a smooth compact orbifold with boundary $\partial\orbify{X}$. Assume the singular set of $\orbify{X}$ has  codimension greater than 1. Then $\partial\orbify{X}$ is not a smooth orbifold retract of $\orbify{X}$.
\end{corollary}

\begin{remark} Orbifolds can be regarded as rational homology manifolds and corollary~\ref{NoRetractionThmOrb} provides a nice subclass of such rational homology manifolds for which a Borsuk no retraction result holds.
\end{remark}

In light of example~\ref{1-OrbifoldRetraction}, one might suspect that the existence of codimension~1 strata is enough to guarantee a retraction to the boundary. The following two examples show that this is not the case.

\begin{example}\label{PairOfPants}[A Pair of Pants with Mirror]
\end{example}

\begin{figure}[ht]
  \centering
  \includegraphics[width=120pt]{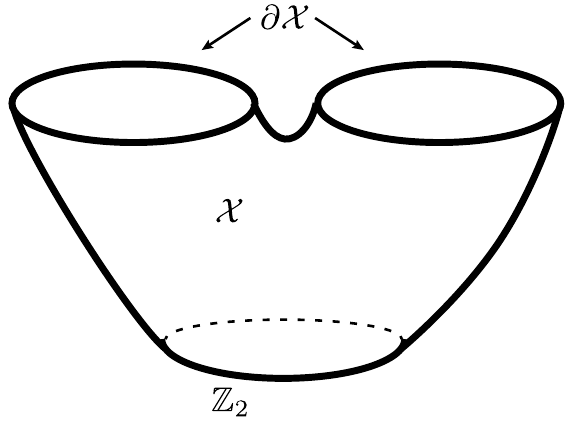}
  \caption{An orbifold $\orbify{X}$ with only codimension 1 strata that does not retract to $\partial\orbify{X}$}
  \label{Codim1StrataNotRetractableEx}
\end{figure}

\begin{example}\label{KnotComplement}[A Knot Complement] Consider the closed 3-ball $D^3$ and let $K$ denote any embedded tubular neighborhood of a knot in the interior of $D^3$. The boundary of $D^3-K$ is the disjoint union $S^2\coprod T^2$ of a 2-sphere and 2-torus. Consider the 3-orbifold with boundary $\orbify{X}$ whose underlying topological space is $D^3-K$ where we consider the $S^2$ (topological) boundary component as a $\mathbb{Z}_2$ mirror. Thus, as an orbifold $\partial\orbify{X}=T^2$. Because $H_1(\partial\orbify{X},\mathbb{Z})\cong\mathbb{Z}^2$ and $H_1(\orbify{X},\mathbb{Z})\cong\mathbb{Z}$ since \orbify{X} is a knot complement, there can be no retraction $r:\orbify{X}\to\partial\orbify{X}$ by elementary homology considerations.
\end{example}

Elaborating on the ideas in the proof of theorem~\ref{NoRetractionToBdy}, we can give hypotheses that guarantee that the preimage of a regular value is, in fact, a \emph{1-manifold} (an orbifold with trivial orbifold structure).

\begin{theorem} Let $\orbify{X}$ be a smooth $n$-dimensional orbifold with boundary and $\orbify{P}$ a smooth orbifold with $\dim\orbify{P}=n-1$. Suppose that $p\in\orbify{P}$ is a regular value for a smooth (complete) orbifold map $\starfunc{f}:\orbify{X}\to\orbify{P}$. This will happen, for example, if $\starfunc{f}$ is surjective. Let $\orbify{S}=f^{-1}(p)$. Suppose further that for $x\in\orbify{S}$, $\Gamma_{x,\orbify{X}}$ has no index 2 subgroups acting on $\R^n$ as $\R^{n-1}\oplus\R$ with trivial action on the $\R$ factor. Then $\orbify{S}$ is a compact 1-manifold with an even number of boundary points.
\end{theorem}

\begin{proof} As before, by theorem~\ref{PreimageThmBdy}, $\orbify{S}$ is a compact 1-orbifold and thus is a disjoint union of 1-orbifolds of type (a)-(d). The goal is to show that cases (c) and (d) do not occur. To this end, suppose a component \orbify{C} of $\orbify{S}$ is of type (c) or (d) and choose one of the points $z\in\orbify{C}$ where $\Gamma_{z,\orbify{C}}=\mathbb{Z}_2$. At this point, the kernel $G_{z,\orbify{X}}$ of the quotient homomorphism $\Gamma_{z,\orbify{X}}\to\Gamma_{z,\orbify{C}}$ has index 2 and acts on $\R^n=\R^{n-1}\oplus\R$ trivially on the $\R$ factor. 
By assumption, no such points $z\in\orbify{X}$ exist and we have our desired contradiction.
We conclude, therefore, that $\orbify{S}$ is a compact 1-manifold with an even number of boundary points.
\end{proof}


\bib{MR2359514}{book}{
  author={Adem, Alejandro},
  author={Leida, Johann},
  author={Ruan, Yongbin},
  title={Orbifolds and stringy topology},
  series={Cambridge Tracts in Mathematics},
  volume={171},
  publisher={Cambridge University Press},
  place={Cambridge},
  date={2007},
  pages={xii+149},
  isbn={978-0-521-87004-7},
  isbn={0-521-87004-6},
  review={\MR {2359514}},
}

\bib{Borsuk-1931}{article}{
  author={Borsuk, Karol},
  title={Sur les r\'etractes},
  journal={Fund. Math.},
  volume={17},
  date={1931},
  pages={2--20},
}

\bib{MR2523149}{article}{
  author={Borzellino, Joseph E.},
  author={Brunsden, Victor},
  title={A manifold structure for the group of orbifold diffeomorphisms of a smooth orbifold},
  journal={J. Lie Theory},
  volume={18},
  date={2008},
  number={4},
  pages={979--1007},
  issn={0949-5932},
  eprint={arXiv:math.DG/0608526},
}

\bib{BB_Stratified}{article}{
  author={Borzellino, Joseph E.},
  author={Brunsden, Victor},
  title={The stratified structure of smooth orbifold mappings},
  eprint={arXiv:0810.1070},
  status={preprint},
  year={2008},
}

\bib{MR1950941}{article}{
  author={Chen, Weimin},
  author={Ruan, Yongbin},
  title={Orbifold Gromov-Witten theory},
  conference={ title={Orbifolds in mathematics and physics}, address={Madison, WI}, date={2001}, },
  book={ series={Contemp. Math.}, volume={310}, publisher={Amer. Math. Soc.}, place={Providence, RI}, },
  date={2002},
  pages={25--85},
  review={\MR {1950941 (2004k:53145)}},
}

\bib{Kankaanrinta}{article}{
  author={Kankaanrinta, Marja},
  title={A Subanalytic Triangulation Theorem For Real Analytic Orbifolds},
  eprint={arXiv:1105.0209},
  status={preprint},
  year={2011},
}

\bib{MR1487640}{book}{
  author={Milnor, John W.},
  title={Topology from the differentiable viewpoint},
  series={Princeton Landmarks in Mathematics},
  note={Based on notes by David W. Weaver; Revised reprint of the 1965 original},
  publisher={Princeton University Press},
  place={Princeton, NJ},
  date={1997},
  pages={xii+64},
  isbn={0-691-04833-9},
  review={\MR {1487640 (98h:57051)}},
}

\bib{MR1466622}{article}{
  author={Moerdijk, I.},
  author={Pronk, D. A.},
  title={Orbifolds, sheaves and groupoids},
  journal={$K$-Theory},
  volume={12},
  date={1997},
  number={1},
  pages={3--21},
  issn={0920-3036},
  review={\MR {1466622 (98i:22004)}},
  doi={10.1023/A:1007767628271},
}

\bib{MR0073104}{book}{
  author={Montgomery, Deane},
  author={Zippin, Leo},
  title={Topological transformation groups},
  publisher={Interscience Publishers, New York-London},
  date={1955},
  pages={xi+282},
  review={\MR {0073104 (17,383b)}},
}

\bib{MR0079769}{article}{
  author={Satake, I.},
  title={On a generalization of the notion of manifold},
  journal={Proc. Nat. Acad. Sci. U.S.A.},
  volume={42},
  date={1956},
  pages={359--363},
  issn={0027-8424},
  review={\MR {0079769 (18,144a)}},
}

\bib{MR0095520}{article}{
  author={Satake, Ichir{\^o}},
  title={The Gauss-Bonnet theorem for $V$-manifolds},
  journal={J. Math. Soc. Japan},
  volume={9},
  date={1957},
  pages={464--492},
  issn={0025-5645},
  review={\MR {0095520 (20 \#2022)}},
}

\bib{Thurston78}{manual}{
  author={Thurston, William},
  title={The geometry and topology of 3--manifolds},
  organization={Princeton University Mathematics Department},
  date={1978},
  note={Lecture Notes},
}



\begin{bibdiv} 
\begin{biblist}

\bibselect{ref}
 
\end{biblist} 
\end{bibdiv}
\end{document}